\documentclass{article}
\usepackage{makeidx}
\usepackage{amssymb}
\usepackage{amsmath, amsthm}
\usepackage{color}
\usepackage{hyperref}
\def\QED{\nolinebreak \hspace*{\fill} \ensuremath{\square}}

\newcommand{\comments}[1]{}
\newcommand{\N}{\mathbb{N}}
\newcommand{\Q}{\mathbb{Q}}
\newcommand{\R}{\mathbb{R}}
\newcommand{\cs}{2^\omega}
\newcommand{\fs}{2^{<\omega}}
\newcommand{\MLR}{\text{MLR}}

\newcommand{\restrict}{\upharpoonright}
\newcommand{\n}{\! \! \!}
\newcommand{\wn}{\widehat{\nu}}
\newcommand{\M}{\mathcal{M}}

\newtheorem{theorem}{Theorem}
\newtheorem{lemma}[theorem]{Lemma}
\newtheorem{remark}[theorem]{Remark}
\newtheorem{definition}[theorem]{Definition}

\begin{document}
	\title{
		How much randomness is needed for statistics?
	}
	\author{Bj\o rn Kjos-Hanssen\\ Antoine Taveneaux\\ Neil Thapen}
	\maketitle
	\begin{abstract}
		In algorithmic randomness, when one wants to define a randomness notion
		with respect to some non-computable measure $\lambda $, a choice needs to be made.
		One approach is to allow randomness tests to access the measure $\lambda $ as
		an oracle (which we call the ``classical approach''). The other approach is the
		opposite one, where the randomness tests are completely effective and do
		not have access to the information contained in $\lambda $ (we call this
		approach ``Hippocratic''). While the Hippocratic approach is in general
		much more restrictive, there are cases where the two coincide. The first
		author showed in 2010 that in the particular case where the notion of
		randomness considered is Martin-L\"of randomness and the measure $\lambda $ is
		a Bernoulli measure, classical randomness and Hippocratic randomness
		coincide. In this paper, we prove that this result no longer holds for
		other notions of randomness, namely computable randomness and
		stochasticity.
	\end{abstract}

	\section{Introduction}
		In algorithmic randomness theory we are interested in
		which almost sure properties of an infinite sequence of bits are effective or computable in some sense.
		Martin-L\"of defined randomness with respect to the
		uniform fair-coin measure $\mu$ on $\cs$ as follows.

		\begin{quote}
			A sequence $X\in \cs $ is \emph{Martin-L\"of random} if we have
			$ X \not\in \bigcap_{n \in \N} \mathcal{U}_n$
			for every sequence of uniformly $\Sigma^0_1 $ (or effectively open) subsets of $\cs $ such that $\mu (\mathcal{U}_n) \leq 2^{-n}$.
		\end{quote}

		Now if we wish to consider Martin-L\"of randomness for a Bernoulli measure~$\mu_p $
		(that is, a measure such that the $i^{\text{th}}$ bit is the result of a
		Bernoulli trial with parameter~$p\in [0,1]$),
		we have two possible ways to extend the previous definition.

		The first option is to consider $p$ as an oracle (with an oracle $p$ we can compute~$\mu_p $) and relativize everything to this oracle.
		Then $X$ is $\mu_p $-Martin-L\"of random if
		for every sequence ${(\mathcal{U}_n)}_{n\in \N}$ of uniformly $\Sigma_1^0[p]$ sets such that
		$\mu_p (\mathcal{U}_n) \leq 2^{-n}$ we have $ X \not\in \bigcap_{n \in \N} \mathcal{U}_n$.
		We will call this approach the \emph{classical}\footnote{
			The classical approach has actually
			two approaches. Reimann and Slaman \cite[arXiv:0802.2705, Definition 3.2.]{RS}
			defined a real $x$ to be $\mu$-random if, for some oracle $z$ computing $\mu$, the real $x$ is
			$\mu$-random relative to $z$. Levin \cite{MR0414222} and G\'acs \cite{MR2159646} use a uniform test,
			which is a left-c.e. function $u : 2^\omega \times M(2^\omega)\rightarrow [0,\infty]$
			such that $\int u(x, \mu)d\mu \le 1$ for all $\mu$ where $M(2^\omega)$ is the space of probability
			measures on $2^\omega$. Since there is a universal uniform test $u_0$, define $x$ to be $\mu$-random if
			$u_0(x, \mu) < \infty$. Day and Miller \cite{MR3042595}
			showed that these approaches actually coincide.
		}
		notion of Martin-L\"of randomness relative to $\mu_p$.

		Another option is to keep the measure $\mu_p $ ``hidden'' from the process which describes the sequence $(\mathcal{U}_n)$.
		One can merely replace $\mu$ by $\mu_p $ in Martin-L\"of's definition but still require $(\mathcal{U}_n)$ to be uniformly $\Sigma_1^0 $ in the unrelativized sense.
		This notion of randomness was introduced by Kjos-Hanssen \cite{K10} who called it \textit{Hippocratic randomness};
		Bienvenu, Doty and Stephan \cite{BDS09} used the term \textit{blind randomness}.
		
		Kjos-Hanssen showed that for Bernoulli measures, Hippocratic and classical randomness coincide in the case of Martin-L\"of randomness.
		Bienvenu, G\'acs, Hoyrup, Rojas and Shen \cite{BGHRS11} extended Kjos-Hanssen's result to other classes of measures.
		Here we go in a different direction and consider weaker randomness notions, such as computable randomness and stochasticity.
		We discover the contours of a dividing line for the type of betting strategy that is needed
		in order to render the probability distribution superfluous as a computational resource.

		We view \emph{statistics} as the discipline concerned with determining the underlying probability distribution $\mu_p$ by looking at the bits of a random sequence.
		In the case of Martin-L\"of randomness it is possible to determine $p$ (\cite{K10}), and therefore Hippocratic randomness and classical randomness coincide.
		In this sense, Martin-L\"of randomness is sufficient for statistics to be possible,
		and it is natural to ask whether smaller amounts of randomness, such as computable randomness, are also sufficient.

		\paragraph{Notation} Our notation generally follows Nies' monograph \cite{Nies_2009}.
		We write $2^n$ for ${\{0, 1\}}^n$, and for sequences $\sigma \in 2^{\le \omega}$ we will
		also use $\sigma$ to denote the real with binary expansion $0.\sigma$, that is,
		the real $\sum_{i=1}^{ \infty} \sigma (i) 2^{- i}$. We use
		$\varepsilon$ to denote the empty word, $\sigma(n)$ for the ${n}^{\text{th}}$ element of
		a sequence and $\sigma \restrict n$ for the sequence formed by the first $n$
		elements. For sequences $\rho, \sigma$ we write $\sigma \prec \rho$ if $\sigma$ is a
		proper prefix of $\rho$
		and denote the concatenation of $ \sigma $ and $\rho $ by $\sigma. \rho $ or simply
		$\sigma \rho$.
		Throughout the paper we set $n'=n(n-1)/2$.

		\subsection{Hippocratic martingales}

			Formally a martingale is a function $\M : \fs \rightarrow \mathbb{R}^{\ge 0}$ satisfying
			\[
				\mathcal{M} (\sigma) =\frac{\mathcal{M}(\sigma 0) + \mathcal{M}(\sigma 1)}{2}.
			\]
			Intuitively, such a function
			arises from a betting strategy for a fair game played with an unbiased coin
			(a sequence of Bernoulli trials with parameter $1/2$).
			In each round of the game we can choose our stake, that is,
			how much of our capital we will bet,
			and whether we bet on heads ($1$) or tails ($0$). A coin is tossed, and if we bet correctly
			we win back twice our stake.

			Suppose that our betting strategy
			is given by some fixed function $S$ of the history $\sigma$ of the game
			up to that point. Then it is easy
			to see that the function
			$\M(\sigma)$ giving our capital after a play $\sigma$ satisfies the above equation.
			On the other hand, from any $\M$ satisfying the equation we can recover a
			corresponding strategy $S$.

			More generally, consider a biased coin which comes
			up heads with probability $p \in (0,1)$. In a fair game played
			with this coin, we would expect to win back $1/p$ times our stake if
			we bet correctly on heads, and $1/(1-p)$ times our stake if
			we bet correctly on tails. Hence we define a $p$-martingale
			to be a function satisfying
			\[
				\M (\sigma) = p \M (\sigma 1) + (1-p) \M (\sigma 0).
			\]
			We can generalize this further, and for any
			probability measure $\mu$ on $\cs$ define a \emph{$\mu$-martingale} to be a function
			satisfying
			\[
				\mu (\sigma) \M (\sigma) = \mu (\sigma 1) \M (\sigma 1) +
				\mu (\sigma 0) \M (\sigma 0).
			\]

			For the Bernoulli measure with parameter $p$,
			we say that a sequence $X \in \cs$ is \emph{$p$-computably random}
			if for every total, $p$-computable $p$-martingale $\M$,
			the sequence ${(\M(X \restrict n))}_n$ is bounded.

			This is the classical approach to $p$-computable randomness.
			Under the Hippocratic approach, the bits of the parameter $p$ should not be available
			as a computational resource. The obvious change to the definition would be to
			restrict to $p$-martingales $\M$ that are computable without an oracle for $p$.
			However this does not give a useful definition,
			as $p$ can easily be recovered from any non-trivial
			$p$-martingale.
			Instead we will define $\mu_p$-Hippocratic computable martingales in terms of
			their \emph{stake function} (or \emph{strategy}) $S$.

			We formalize $S$ as a function
			$\fs \rightarrow [-1, 1] \cap \mathbb{Q}$.\footnote{
				The
				restriction to $\mathbb{Q}$ is justified by the fact that we can restrict to $\mathbb{Q}$ in
				the definition of computable randomness.
			} The absolute value $|S(\sigma)|$
			gives the fraction of our capital we put up as our stake, and
			we bet on $1$ if $S(\sigma) \ge 0$ and on $0$ if $S(\sigma) <0$.
			Given $\alpha \in (0,1)$, the $\alpha$-martingale $\M^\alpha$ arising from $S$
			is then defined inductively by
			\begin{align*}
				\M^\alpha (\varepsilon) &= 1 \\
				\M^\alpha (\sigma 1) &= \M^\alpha (\sigma)
				\Big( 1 - |S(\sigma)| + \frac{|S(\sigma)|}{\alpha} 1_{\{S(\sigma) \geq 0 \}} \Big)\\
				\M^\alpha (\sigma 0) &= \M^\alpha (\sigma)
				\Big( 1 - |S(\sigma)| + \frac{|S(\sigma)|}{1-\alpha} 1_{\{S(\sigma) < 0 \}} \Big)
			\end{align*}
			where, for a formula $T$, we use the notation $1_{\{T\}}$ to mean the function which
			takes the value $1$ if $T$ is true and $0$ if $T$ is false.

			We define a \emph{$\mu_p$-Hippocratic computable martingale} to be a
			$p$-martingale $\M^p$ arising from some total computable (without access to $p$)
			stake function $S$. We say that a sequence $X \in \cs$ is \emph{$\mu_p$-Hippocratic computably random}
			if for every $\mu_p$-Hippocratic computable martingale $\M$,
			the sequence ${(\M(X \restrict n))}_n$ is bounded.

			In Section \ref{computrandomness_sect} below
			we show that for all $p \in \MLR$ the set of $\mu_p$-Hippocratic computably random sequences is strictly bigger than the set of $p$-computably random sequences.
			More precisely, we show that we can compute a sequence $Q \in \cs $ from $p$ such that $Q$ is $\mu_p$-Hippocratic computably random.
			In a nutshell, the proof works as follows. We use the number $p$ in two ways.
			To compute the $i^{\text{th}}$ bit of $Q$, the first $i$ bits of $p$ are treated as a parameter
			$r= 0 . p_1 \dots p_{i}$, and we pick the $i^{\text{th}}$ bit of $Q$ to look like it has been chosen at random in a Bernoulli trial with bias $r$.
			To do this, we use some fresh bits of $p$ (which have not been used so far in the construction of $Q$)
			and compare them to $r$,
			to simulate the trial. Since these bits of $p$ were never used before, if we know only the first $i-1$ bits of $Q$ they appear random,
			and thus the $i^{\text{th}}$ bit of $Q$ indeed appears to be chosen at random with bias $r$.
			Since $r=0.p_1 p_2 \dots p_i$ converges quickly to $p$,
			\footnote{By the Law of the Iterated Logarithm and since $2^{-n}=o(1/\sqrt{n\log\log n})$, this convergence is faster than the deviations created by statistical noise in a real sequence of Bernoulli trials with parameter $p$.}
			we are able to show that $Q$ overall looks $p$-random as long as we do not have access to $p$, in other words,
			that $Q$ is $\mu_p$-Hippocratic computably random.

		\subsection{Hippocratic stochasticity and KL randomness}

			In Section
			\ref{sect_hypocratic_stoc} we consider
			another approach to algorithmic randomness, known as stochasticity.
			It is reasonable to require that a random sequence satisfies the law of large numbers,
			that is, that the proportion of $1$s in the sequence converges to the bias $p$.
			But, for an unbiased coin, the string
			\[
				010101010\dots
			\]
			satisfies this law but is clearly not random.
			Following this idea, we say that a sequence $X$ is
			\emph{$p$-Kolmogorov--Loveland stochastic} (or \emph{$\mu_p$-KL stochastic}) if
			there is no $p$-computable way to select infinitely many bits from $X$,
			where we are not allowed to know the value of a bit before we select it,
			without the selected sequence satisfying the law of large numbers
			(see Definition \ref{def_KL_Stoc} for a formal approach).

			For this paradigm the Hippocratic approach is clear:
			we consider only selection functions which are computable without an oracle for $p$.
			We show in Theorem \ref{Thm_hypo_KL_stoc} that for $p \in \Delta_2^0 \cap \MLR$
			there exists a sequence $Q$ which is $\mu_p$-Hippocratic KL stochastic but not $\mu_p$-KL stochastic.
			Again we use $p$ as a random bit generator and create a sequence $Q$ that appears random
			for a sequence of Bernoulli trials, where the bias of the
			$i^{\text{th}}$ trial is $q_i$ for a certain sequence ${(q_i)}_i$ converging to $p$.
			Intuitively, the convergence is so slow
			that it is impossible to do (computable) statistics with $Q$ to recover $p$, and
			we are able to show that without access to $p$ the sequence $Q$ is
			$\mu_p$-KL stochastic.

			At the end of Section
			\ref{sect_hypocratic_stoc} we consider another notion,
			Kolmogorov--Loveland randomness. We give a simple argument to show
			that if we can compute $p$ from every $\mu_p$-Hippocratic KL random sequence, then
			the $\mu_p$-Hippocratic KL random sequences and the $\mu_p$-KL random sequences are the
			same (and vice versa).

	\section{Computable randomness}\label{computrandomness_sect}

		In this section we show that for any Martin-L\"of random bias $p$,
		$p$-computable randomness is a stronger notion
		than $\mu_p$-Hippocratic computable randomness.

		\begin{theorem}\label{computrandomness_thm}
			Let $\alpha \in \MLR$. There exists a sequence $Q\in \cs$,
			computable in polynomial time from $\alpha$, such that
			$Q$ is $\mu_\alpha$-Hippocratic computably random.
		\end{theorem}

		Before giving the proof, we remark that
		a stronger version of the theorem is true:
		the sequence $Q$ is in fact
		$\alpha$-Hippocratic partial computably random
		(meaning that we allow the martingale to be a partial computable function; see \cite[Definition 7.4.3]{Downey:2010:ARC:1205766}).

		Also, a sceptic could (reasonably) complain that it is not
		really natural for us to make bets without any idea about our current capital. However
		if we add an oracle to give the integer part of our capital at each step (or even an approximation with accuracy $2^{-n}$ when we bet on the $n^{\text{th}}$ bit),
		Theorem \ref{computrandomness_thm} remains true and the proof is the same.
		In the same spirit we could object that it is more natural to have a stake function giving the amount of our bet
		(to be placed only if we have a capital large enough)
		and not the proportion of our capital. For this definition of a Hippocratic computable martingale, similarly the theorem remains true and the proof is the same.

		\begin{proof}
			Let $\alpha \in \MLR$. Then $\alpha$ is not rational and
			cannot be represented by a finite binary sequence and we can suppose that $0< \alpha < 1/2$.
			Recall that $n'=n(n-1)/2$ and that we freely identify a sequence $X$ (finite or infinite) with the real number with the binary expansion $0.X$.

			The proof has the following structure.
			First, we describe an algorithm to compute a sequence $Q$ from $\alpha $.
			To compute each bit $Q_n$ of $Q$ we will use a finite initial segment of $\alpha$
			as an approximation of $\alpha$,
			and we will compare this with some other fresh bits of $\alpha$ which we treat as though
			they are produced by a random bit generator.
			In this way $Q_n$ will approximate the outcome of a Bernoulli trial
			with bias $\alpha$.

			Second, we will suppose for a contradiction that there is an $\alpha$-Hippocratic computable martingale
			(that is, a martingale that arises from a stake function computable without $\alpha$)
			such that the capital of this martingale is not bounded on $Q$.
			We will show that we can use this stake function to construct
			a Martin L\"{o}f test ${(U_n)}_n$ such that $\alpha$ does not pass this test.

			So let $Q = Q_1 Q_2 \dots$ be defined by the condition that:
			\[
				Q_n = \left\{
				\begin{array}{ll}
				0 & \text{ if } 0.\alpha_{n' + 1} \dots \alpha_{n' + n}\geq 0. \alpha_{1} \dots \alpha_{n}, \\
				1& \text{ otherwise.}
				\end{array}
				\right.
			\]
			We can compute $Q$ in polynomial time from $\alpha$,
			as we can compute each bit $Q_n$ in time $O(n^2)$.

			Now let $S : \fs \rightarrow \Q \cap [-1, 1]$ be a computable stake function.
			We will write $\M^X$ for the $X$-martingale arising from $S$.
			Suppose for a contradiction that
			\[
				\limsup_{n\rightarrow \infty} \mathcal{M}^{\alpha} (Q \restrict n) = \infty .
			\]

			Our goal is to use this to define a Martin-L\"of test which $\alpha$ fails.
			The classical argument (see \cite[Theorem 6.3.4]{Downey:2010:ARC:1205766})
			would be to consider the sequence of sets
			\[
				V_j = \{ X \in \cs | \exists n ~~ \mathcal{M}^{\alpha} (X \restrict n) > 2^j \},
			\]
			but without oracle access to $\alpha$ this
			is not $\Sigma_1^0$, and does not define a Martin-L\"of test.
			However it turns out that we can use a similar sequence of sets, based on the idea
			that, although we cannot compute $\M^\alpha$ precisely,
			we can approximate it using the approximation
			$\alpha_1 \dots \alpha_{n'}$ of $\alpha$.
			For this we will use the following lemma, giving, roughly speaking, a modulus of continuity for the map
			$(\alpha, X) \mapsto \mathcal{M}^{\alpha} (X)$.
			The proof
			is rather technical and we postpone it
			to later in this section.

			\begin{lemma} \label{approxlemmamartin}
				For $0<\alpha<1$,
				there exists $m \in \N$ such that $2^{-m} < \alpha\restrict m'$, and such that if $\sigma \succcurlyeq (\alpha \restrict m')$ and
				$\tau \succcurlyeq (\alpha \restrict m')$ then for all $\eta \in \fs$ and all
				$n \geq m$ we have:
				\[
					   \text{if }~ 0< \tau - \sigma < 2^{-n'} \text{ and }~ |\eta |\leq n+1
					\text{ then }~ |\mathcal{M}^{\sigma} (\eta) - \mathcal{M}^\tau (\eta) | \leq 2^{-n}.
				\]
			\end{lemma}

			Let $m$ be given by Lemma \ref{approxlemmamartin} and let $\rho$ be $\alpha \restrict m'$, so $2^{-m}<\rho$.
			Let $\Gamma :2^{\leq \omega} \rightarrow 2^{\leq \omega}$ be the operator
			which converts $\alpha_1 \dots \alpha_{n'}$ into $Q_1 \dots Q_n $. That is,
			$\Gamma ( \alpha_1 \dots \alpha_k) =Q_1 \dots Q_n$ where $n$ is the
			biggest integer such that $n'\leq k$.
			This notation naturally extends to infinite sequences so we may write
			$\Gamma(\alpha) = Q$.
			We consider the uniform sequence of $\Sigma_1^0$ sets
			\[
				U'_j =\{
					X_1 \dots X_{k'} |\rho\preccurlyeq X_1 \dots X_{k'}
					\text{ and }
					\mathcal{M}^{X_1 \dots X_{k'}} ( \Gamma (X_1 \dots X_{k'})) > 2^{j}
				\}.
			\]
			We let $U_j$ denote the set of infinite sequences with a prefix in $U'_j$.
			By Lemma \ref{approxlemmamartin},
			\[
				|
					\mathcal{M}^{\alpha}( \Gamma (\alpha_1 \dots \alpha_{k'})) -
					\mathcal{M}^{\alpha_1 \dots \alpha_{k'}} ( \Gamma (\alpha_1 \dots \alpha_{k'}))
				|
				< 2^{-k} \leq 1
			\]
			for all sufficiently large $k$.
			Since $\M^\alpha$ increases unboundedly on $Q = \Gamma(\alpha)$
			it follows that $\alpha \in U_j$ for all $j$.

			To show that $(U_j)$ is a Martin-L\"of test, it remains to show
			that the measure of $U_j$ is small.
			Since $\sigma \mapsto \mathcal{M}^{\sigma} (\sigma)$ is almost a $\alpha$-martingale,
			where $\sigma$ runs over the prefixes of $\alpha$, we will use a lemma
			similar to the Kolmogorov inequality (see \cite[Theorem 6.3.3]{Downey:2010:ARC:1205766}).
			Again we postpone the proof to later in this section.

			\begin{lemma}\label{almost_Kolmogorov_inequality}
				For any number $n \ge m$, any extension
				$\sigma \succcurlyeq \rho$ of length $n'$ and
				any prefix-free set
				$Z \subseteq \bigcup_{k \in \N} {\{ 0, 1 \}}^{k'}$ of extensions of $\sigma $, we have
				\[
					\sum_{\tau \in Z} 2^{-|\tau |} \mathcal{M}^{\tau} (\Gamma (\tau))
					\leq 2^{-|\sigma |} e^2 \left[ 1 +
					\mathcal{M}^{\sigma} (\Gamma (\sigma)) \right].
				\]
			\end{lemma}

			Now fix $j$ and let $W_j$ be a prefix-free subset of $U'_j$ with the property that
			the set of infinite sequences with a prefix in $W_j$ is exactly $U_j$.
			Then by the definition of $U'_j$, if $\tau \in W_j$ then
			$\mathcal{M}^{\tau} (\Gamma (\tau)) \geq 2^j$.
			Hence by Lemma \ref{almost_Kolmogorov_inequality} we have:
			\[
				\mu ( U_j) = \sum_{\tau \in W_j} 2^{-|\tau |}
				\leq \sum_{\tau \in W_j} \frac{\mathcal{M}^{\tau} (\Gamma (\tau))}{2^j} 2^{-|\tau |}
				\leq \frac{ 2^{-|\rho |} e^2 \left(1 + \M^{\rho} (\Gamma (\rho)) \right)}{2^j}.
			\]
			Since $2^{-|\rho |} \left(1 + \M^{\rho} (\Gamma (\rho)) \right)$ is constant,
			this shows that $(U_j)$ is a Martin-L\"of test.
			As $\alpha \in \bigcap_j U_j$ it follows that $\alpha \not\in \MLR$.
			This is a contradiction.
			\QED
		\end{proof}

		Notice that this proof makes use of the fact that in our betting strategy
		we have to proceed monotonically from left to right through the string,
		making a decision for each bit in turn as we come to it.
		This is why our construction is able to use $\alpha$ as a random bit generator,
		because at each step it can use bits that were not used to compute the previous bits of $Q$.
		Following this idea the question naturally arises:
		if we are allowed to use a non-monotone strategy, then are the classical and Hippocratic
		random sequences the same? We explore this question in Section \ref{sect_hypocratic_stoc}.

		\medskip

		We now return to the postponed proofs.
		We will need a couple of technical lemmas,
		the first one is aiding Lemma \ref{approxlemmamartin} in giving (roughly speaking) a modulus of continuity for the map
		$(\alpha, X) \mapsto \mathcal{M}^{\alpha} (X)$.

		\begin{lemma}\label{approxlemma}
			Let $\epsilon >0$.
			Then there exists $r \in \N$ such that for all $k$ with $2^{-k}\epsilon^{-2}<1$,
			for all $\alpha, \beta \in 2^{\le \omega}$ with $\epsilon < \alpha < \beta < 1- \epsilon$ and for all
			non-empty $\sigma \in \fs$,
			\[
				0 < \beta - \alpha < 2^{-k}
				\quad
				\Longrightarrow
				\quad ~| \mathcal{M}^{\alpha} (\sigma) - \mathcal{M}^{\beta} (\sigma) |< 2^{-k+r |\sigma |}.
			\]
		\end{lemma}
		\begin{proof}
			Since $0 < \epsilon < \alpha < \beta < \alpha +2^{-k}$,
			\[
				\frac{1}{\alpha + 2^{-k}} < \frac{1}{\beta} < \frac{1}{\alpha} < \frac{1}{\epsilon},
			\]
			and hence
			\[
				0 < \frac{1}{\alpha} - \frac{1}{\beta}
				< \frac{1}{\alpha} - \frac{1}{\alpha + 2^{-k}}
				= \frac{2^{-k}}{\alpha (\alpha + 2^{-k})}
				< \frac{2^{-k}}{\alpha^2}
				< \frac{2^{-k}}{\epsilon^2}.
			\]
			It follows, since $|S(X)|\le1$, that
			\[
				0 \! \leq \!
				\Big( 1- |S(\sigma)| + \frac{|S(\sigma)|}{\alpha} 1_{\{S(\sigma) \geq 0 \}} \Big) \! - \!
				\Big( 1 - |S(\sigma)| + \frac{|S(\sigma)|}{\beta} 1_{\{S(\sigma) \geq 0 \}} \Big)
				\! \leq {2^{-k}} \epsilon^{-2}
			\]
			and symmetrically, since $0 < \epsilon < 1-\beta < 1-\alpha < (1-\beta)+2^{-k}$, also that
			\[
				0 \! \leq \!
				\Big( 1 - |S(\sigma)| + \frac{|S(\sigma)|}{1- \beta} 1_{\{S(\sigma) < 0 \}} \Big) \! - \!
				\Big( 1 - |S(\sigma)| + \frac{|S(\sigma)|}{1- \alpha} 1_{\{S(\sigma) < 0 \}} \Big)
				\! \leq {2^{-k}} \epsilon^{-2}.
			\]
			Hence if we write $R_i^X$ for
			$\frac{\mathcal{M}^{X} (\sigma \restrict i)}
			{\mathcal{M}^{X} (\sigma \restrict (i-1))}$
			(with the convention $0/0=0$) we have for all $i \le |\sigma|$ that
			$|R_i^\alpha-R_i^\beta|
			< 2^{-k} \epsilon^{-2}$.
			Furthermore, take $s$ to be a positive integer such that $2^s >1+ 1 / \epsilon$.
			Then we know that $R_i^\alpha$ and $R_i^\beta$ are both
			always smaller than $2^s$.

			\sloppypar
			We can now bound $|\mathcal{M}^{\alpha} (\sigma) - \mathcal{M}^{\beta} (\sigma) |$.
			Consider the case when
			\mbox{ $\mathcal{M}^{\alpha} (\sigma) \geq \mathcal{M}^{\beta} (\sigma)$}
			(the other case is symmetrical).
			Then, writing $n$ for $|\sigma|$,
			\begin{align*}
				\mathcal{M}^{\alpha} (\sigma) - \mathcal{M}^{\beta} (\sigma)
				&= \prod_{i=1}^n R_i^\alpha - \prod_{i=1}^n R_i^\beta \\
				&\le \prod_{i=1}^n (R_i^\beta + 2^{-k} \epsilon^{-2}) - \prod_{i=1}^n R_i^\beta \\
				&= \Biggl[ \prod_{i=1}^n R_i^\beta
				+ \sum_{\substack{Z \subseteq [1, n]\\|Z|<n}}
				(2^{-k} \epsilon^{-2})^{n-|Z|}\prod_{i \in Z} R_i^\beta \Biggr]
				- \prod_{i=1}^{n} R_i^\beta\\
				&\le 2^n (2^{-k} \epsilon^{-2}) (2^s)^n,
			\end{align*}
			where for the last inequality we are assuming that $k$ is large enough that
			\mbox{$2^{-k} \epsilon^{-2} <1$}. The result follows.
			\QED
		\end{proof}

		\begin{lemma}\label{convproductlemma}
			For $s \in \R$, $s>0$ we have
			$\prod_{n=1}^{\infty} (1+s 2^{-n}) < e^{s}$.
		\end{lemma}
		\begin{proof}
			It is enough to show that
			\[
				\sum_{n=1}^{\infty} \ln \left(\frac{2^n +s}{2^n} \right)=\sum_{n=1}^{\infty} \left[ \ln ( 2^n +s) - \ln (2^n) \right] <s.
			\]
			The derivative of $\ln$ is the decreasing function $x \mapsto 1/x$
			so by the Mean Value Theorem we have that
			$\ln ( 2^n + s) - \ln (2^n) < s/ 2^n$, which gives the inequality.
			\QED
		\end{proof}

		We are now able to prove the lemmas used in the proof
		of Theorem \ref{computrandomness_thm}.

		\medskip

		\noindent{\textbf{Restatement of Lemma \ref{approxlemmamartin}.}}
		\textit{
		For $0<\alpha<1$,
		there exists $m \in \N$ such that $2^{-m} < \alpha\restrict m'$, and such that if $\sigma \succcurlyeq (\alpha \restrict m')$ and
		$\tau \succcurlyeq (\alpha \restrict m')$ then for all $\eta \in \fs$ and all
		 $n \geq m$ we have:
		\[
			\text{if }~ 0< \tau - \sigma < 2^{-n'} \text{ and }~ |\eta |\leq n+1
			\text{ then }~ |\mathcal{M}^{\sigma} (\eta) - \mathcal{M}^\tau (\eta) | \leq 2^{-n}.
		\]}
		\begin{proof}
			Since $0<\alpha<1$, there is an $\epsilon>0$ and an $m_0$ such that
			for all $n\ge m_0$, 
			\[
				\epsilon < \alpha \restrict n' < (\alpha \restrict n')
				+ 2^{-n'}< 1 - \epsilon.
		 	\]
			Let $r$ be as in Lemma \ref{approxlemma} for this $\epsilon$.
			It is clear that we can find $m_1 \in \N$ such that for all $n \ge m_1$,
			\[
				r(n+1)- n'=r(n+1)- \frac{n(n-1)}{2} < -n.
			\]
			Moreover we can find an $m_2$ such that for all $n\ge m_2$,
			\[
				2^{-n'} \epsilon^{-2} < 1.
			\]
			And, since $\alpha > 0$, we can find an $m_3$ such that for all $n\ge m_3$, $2^{-n} < \alpha \restrict n'$.
			
			Let $m = \max\{m_0, m_1, m_2, m_3\}$. Let $\tau$, $\sigma$, $\eta$ and $n$
			satisfy the assumptions of the Lemma.
			We must have $\epsilon < \sigma < \tau < 1-\epsilon$,
			hence by Lemma \ref{approxlemma} with $k := n'$,
			\[
				|\mathcal{M}^{\sigma} (\eta) - \mathcal{M}^\tau (\eta) |
				\leq 2^{-n'+r|\eta|} \leq 2^{-n'+r(n+1)} < 2^{-n}.
			\]
			\QED
		\end{proof}
		\begin{remark}
			As pointed out by an anonymous referee and as is clear from the proof,
			the function that maps $\alpha\in\MLR$ to $m$ in Lemma \ref{approxlemmamartin} is layerwise computable,
			i.e., from the randomness deficiency of $\alpha$ we can compute an $m$ that works.
		\end{remark}

		Now we suppose $\alpha \in \MLR$, $\alpha < 1/2$ and let $m$ be as given by
		Lemma \ref{approxlemmamartin}. We write $\rho$ for $\alpha \restrict m'$, so that $2^{-m}<\rho$.

		\vspace{10pt}

		\noindent{\textbf{Restatement of Lemma \ref{almost_Kolmogorov_inequality}.}}
		\textit{
		For any number $n \ge m$, any extension
		$\sigma \succcurlyeq \rho$ of length $n'$ and
		any prefix-free set
		$Z \subseteq \bigcup_{k \in \N} {\{ 0, 1 \}}^{k'}$ of extensions of $\sigma $, we have
		\[
			\sum_{\tau \in Z} 2^{-|\tau |} \mathcal{M}^{\tau} (\Gamma (\tau))
			\leq 2^{-|\sigma |} e^2 \left[ 1 +
			\mathcal{M}^{\sigma} (\Gamma (\sigma)) \right].
		\]
		}
		\begin{proof}
			It is enough to show this for every finite $Z$.
			We will use induction on the size $p$ of $Z$,
			with our inductive hypothesis that for all $n \ge m$, all extensions
			$\sigma \succcurlyeq \rho$ of length $n'$ and
			all suitable sets $Z$ of size $p$,
			\[
				\sum_{\tau \in Z} 2^{-|\tau |} \mathcal{M}^{\tau} (\Gamma (\tau))
				\leq 2^{-|\sigma |} \left[ \sum_{i=n}^{\infty} 2 e^2 2^{-i} +
				\mathcal{M}^{\sigma} (\Gamma (\sigma)) \prod_{i= n}^{ \infty} \left(1 +2 {\cdot} 2^{-i} \right) \right].
			\]
			Note that by Lemma \ref{convproductlemma} the right hand side
			is bounded by $2^{-|\sigma |} e^2 [ 1 +
			\mathcal{M}^{\sigma} (\Gamma (\sigma))]$ (as long as $n \ge 2$).

			The base case $|Z|=0$ is trivial. Now suppose that the hypothesis is true for all sets of size less
			than or equal to $p$ and suppose that $|Z|=p+1$.
			Let $\nu$ be the longest extension of $\sigma$ which has length of the form $k'$ for some $k \in \N$
			and which
			is such that all strings in $Z$ are extensions of $\nu$.
			Then for each string $\theta$ of length $k$ there are fewer than $p+1$ strings in
			$Z$ beginning with $\nu \theta$. Recall that $|\nu \theta| = k'+k=(k+1)'$ and that
			$\Gamma(\nu)$ and $\Gamma(\nu\theta)$ are strings of length respectively $k$ and $k+1$.
			Applying the inductive hypothesis, we have
			\begin{align*}
				\sum_{\tau \in Z} 2^{-|\tau |} \mathcal{M}^{\tau} (\Gamma (\tau))
				& \le \n \! \!
				\sum_{\theta \in \{ 0, 1 \}^k} \sum_{\substack{\tau \in Z\\ \tau \succcurlyeq \nu \theta}}
				2^{-|\tau |} \mathcal{M}^{\tau} (\Gamma (\tau)) \\
				& \le \n \! \!
				\sum_{\theta \in \{ 0, 1 \}^k} \n \! 2^{-|\nu \theta|}
					\! \left[ \sum_{i=k+1}^{\infty} \! \! 2 e^2 2^{-i} +
					\mathcal{M}^{\nu \theta} (\Gamma (\nu \theta)) \n \prod_{i= k+1}^{ \infty}
				\n \left(1 \! + \! 2 {\cdot} 2^{-i} \right) \right]\\
				& \le \n \! \!
				\sum_{\theta \in \{ 0, 1 \}^k} \n \! 2^{-|\nu \theta|}
					\! \left[ \, \sum_{i=k+1}^{\infty} \! \! 2 e^2 2^{-i} + e^2 2^{-k} +
					\mathcal{M}^{\nu} (\Gamma (\nu \theta)) \n \prod_{i= k+1}^{ \infty}
				\n \left(1 \! + \! 2 {\cdot} 2^{-i} \right) \right]
			\end{align*}
			where for the last inequality we are using that, by Lemma \ref{approxlemmamartin},
			$\mathcal{M}^{\nu \theta} (\Gamma (\nu \theta)) \le \mathcal{M}^{\nu} (\Gamma (\nu \theta)) + 2^{-k}$.
			Rearranging the last line, we get
			\[
				2^{-|\nu |} \left[ \sum_{i=k+1}^{\infty} \n 2 e^2 2^{-i} + e^2 2^{-k}
				+ \left(\sum_{\theta \in {\{0,1\}}^k} \n 2^{-k} \mathcal{M}^{\nu} (\Gamma (\nu \theta)) \right)
				\prod_{i= k+1}^{ \infty} \n \left(1 +2 {\cdot} 2^{-i} \right) \right].
			\]
			Now we will find an upper bound for the term in round brackets.

			We will write $\wn$ for $\nu \restrict k$ and $S$ for $S(\Gamma(\nu))$.
			By the definition of $\Gamma$, if $\theta \le \wn$ (as real numbers) then
			$\Gamma(\nu \theta) = \Gamma(\nu).1$. Hence, by the definition of $\M$ and the fact
			that $\nu \ge \wn$, summing only over $\theta \in {\{ 0, 1 \}}^k$ we have
			\begin{align*}
				\sum_{\theta \le \wn} 2^{-k} \M^{\nu}(\Gamma(\nu\theta))
				&= \wn \, \M^{\nu}(\Gamma(\nu)) \big( 1 - |S| + \frac{1}{\nu} \,|S| \cdot 1_{ \{S \geq 0 \}} \big)\\
				&\le \M^{\nu}(\Gamma(\nu)) \big( \wn \, \big( 1 - |S| \big) + |S| \cdot 1_{ \{S \geq 0 \}} \big).
			\end{align*}
			Observing that $\nu \le \wn + 2^{-k}$ and $1-\nu \ge 1/2$ and that hence
			\[
				\frac{1-\wn}{1-\nu}
				\le \frac{1-\nu + 2^{-k}}{1 -\nu}
				\le 1+ 2 \cdot 2^{-k},
			\]
			we similarly get that
			\begin{align*}
				\sum_{\theta > \wn} 2^{-k} \M^{\nu}(\Gamma(\nu\theta))
				&= (1-\wn) \M^{\nu}(\Gamma(\nu)) \big( 1 - |S| + \frac{1}{1-\nu} \,|S| \cdot 1_{ \{S < 0 \}} \big)\\
				&\le \M^{\nu}(\Gamma(\nu)) \big( (1-\wn) \big( 1 - |S| \big) +
				|S| \cdot 1_{ \{S < 0 \}} + 2 \cdot 2^{-k} \big).
			\end{align*}

			Summing these gives
			\begin{align*}
				\sum_{\theta \in \{ 0, 1 \}^k} \n \! 2^{-k} \mathcal{M}^{\nu} (\Gamma (\nu \theta))
				&\le \M^{\nu}(\Gamma(\nu))
				\big(1 - |S| + |S| \! \cdot \! 1_{ \{S \geq 0 \}} + |S| \! \cdot \! 1_{ \{S < 0 \}}
				+ 2 \cdot 2^{-k} \big)\\
				&= \M^{\nu}(\Gamma(\nu)) (1 + 2 \cdot 2^{-k}).
			\end{align*}
			Combining this with our earlier bound, we now have
			\[
				\sum_{\tau \in Z} 2^{-|\tau |} \mathcal{M}^{\tau} (\Gamma (\tau))
				\leq 2^{-|\nu |} \left[ \sum_{i=k+1}^{\infty} 2 e^2 2^{-i} + e^2 2^{-k} +
				\mathcal{M}^{\nu} (\Gamma (\nu)) \prod_{i= k}^{ \infty} \left(1 +2 {\cdot} 2^{-i} \right) \right].
			\]

			Finally, let $r=k-n$ so that $|\nu| = k' = (n+r)' \ge n'+ nr = |\sigma|+nr$.
			Recall that
			$1-\nu > \nu \succcurlyeq \rho > 2^{-n}$, which means that a $\nu$-martingale can multiply
			its capital by at most $2^n$ in one round.
			Hence, also using Lemma \ref{approxlemmamartin},
			\[
				2^{-|\nu|} \M^{\nu} (\Gamma (\nu))
				\le 2^{-|\sigma|-nr} {(2^n)}^r \M^{\nu}(\Gamma(\sigma))
				\le 2^{-|\sigma|} \big( 2^{-n} + \M^{\sigma}(\Gamma(\sigma)) \big).
			\]
			This gives us the bound
			\[
				2^{-|\sigma|} \left[ \sum_{i=k+1}^{\infty} 2 e^2 2^{-i} + e^2 2^{-k} +
				e^2 2^{-n} +
				\M^{\sigma}(\Gamma(\sigma)) \prod_{i= k}^{ \infty} \left(1 +2 {\cdot} 2^{-i} \right) \right] ,
			\]
			which is less than or equal to
			\[
				2^{-|\sigma|} \left[ \sum_{i=n}^{\infty} 2 e^2 2^{-i} +
				\M^{\sigma}(\Gamma(\sigma)) \prod_{i= n}^{ \infty} \left(1 +2 {\cdot} 2^{-i} \right) \right].
			\]
			This completes the induction.
			\QED
		\end{proof}

	\section{Kolmogorov--Loveland stochasticity and randomness}\label{sect_hypocratic_stoc}

		We define Kolmogorov--Loveland stochasticity and
		show that, in this setting, the Hippocratic and classical approaches give different sets.
		We also consider whether this is true for Kolmogorov--Loveland randomness,
		and relate this to a statistical question.

		Kolmogorov--Loveland randomness and stochasticity has been studied by, among others,
		Merkle \cite{MR2017360},
		Merkle et al. \cite{Merkle05kolmogorov-lovelandrandomness}, and
		Bienvenu \cite{MR2592183}.

		\subsection{Definitions}
			For a finite string $\sigma \in {\{0,1\}}^{n} $,
			we write $\#0(\sigma)$ for $|\{k < n | \sigma (k) = 0 \}|$
			and $\#1(\sigma)$ for $|\{k < n | \sigma (k) = 1 \}|$.
			We write $\Phi (\sigma)$ for $\# 1 (\sigma)/n$, the frequency
			of $1$s in $\sigma$.

			\begin{definition}[Selection function]\label{def_KL_Stoc}
				A KL selection function is a partial function
				\[
					f: \fs \rightarrow \{ \text{scan},\text{select} \} \times \N .
				\]
				We write $f( \sigma)$ as a pair $(s(\sigma), n (\sigma))$
				and in this paper we insist that
				for all $\sigma$ and $\rho \succ \sigma$ we have $n(\rho) \not= n (\sigma)$,
				so that each bit is read at most once.

				Given input $X$,
				we write $(V_f^X)$ for the sequence of strings seen (with bits either scanned or selected)
				by $f$, so that
				\begin{align*}
					V_f^X (0) &= X (n(\varepsilon)) \\
					V_f^X (k+1) &= V_f^X (k). X (n(V_f^X (k))).
				\end{align*}
				We write
				$U_f^X$ for the subsequence of bits selected by $f$.
				Formally $U_f^X$ is the limit of the monotone sequence of strings
				$(T_f^X)$ where
				\begin{align*}
					T_f^X (0) &= \varepsilon \\
					T_f^X (k+1) &= \left\{
					\begin{array}{ll}
						T_f^X (k) & \mbox{ ~ if } s( V_f^X (k))=\text{scan} \\
						T_f^X (k) . n( V_f^X (k)) & \mbox{ ~ if } s( V_f^X (k))=\text{select.}
					\end{array}
					\right.
				\end{align*}
			\end{definition}

			Informally, the function is used to select bits from $X$ in a non-monotone way.
			If $V$ is the string of bits we have read so far, $n(V)$ gives the location of
			the next bit of $X$ to be read. Then ``$s(V)=\text{scan}$'' means that
			we will just read this bit, whereas ``$s(V)=\text{select}$'' means that we will
			add it to our string $T$ of selected bits.

			\begin{definition}[$\mu_p$-KL stochastic sequence]
				A sequence $X$ is $\mu_p$-KL stochastic if for all $p$-computable KL selection functions~$f$ (notice that $f$ can be a partial function)
				such that the limit $U_f^X$ of $(T_f^X)$ is infinite, we have
				\[
					\lim_{k \rightarrow \infty} \Phi (T_f^X (k)) = p.
				\]
				A sequence $X$ is $\mu_p$-Hippocratic KL stochastic if for all KL selection functions~$f$,
				computable without an oracle $p$, such that $U_f^X$
				is infinite, we have
				\[
					\lim_{k \rightarrow \infty} \Phi (T_f^X (k)) = p.
				\]
			\end{definition}

			\begin{definition}[Generalized Bernoulli measure]
				\label{Bernou_generalized}
				A generalized Bernoulli measure $\lambda$ on $\cs$ is determined
				by a sequence $(b_i^\lambda)$ of real numbers in $(0,1)$.
				For each $i$, the event
				\mbox{$\{ X_1 X_2 \dots | X_i =1 \}$}
				has probability $b_i^\lambda$, and these events are all independent.
				In other words, for all finite strings $w$ the set $[w]$
				of strings beginning with $w$ has measure
				\[
					\lambda ([w]) =\prod_{\substack{i<|w|\\w_i=1}} b_i^{\lambda} \prod_{\substack{i<|w|\\w_i=0}} (1- b_i^{\lambda}).
				\]
				We say the measure is computable if the sequence $(b_i^{\lambda})$ is uniformly
				computable.
			\end{definition}

			In some sense a generalized Bernoulli measure
			treats sequences as though they arise from
			a sequences of independent Bernoulli trials
			with parameter $b_i^{\lambda}$ for the $i^{\text{th}}$ bit.

			Recall that
			for a measure $\lambda$,
			a $ \lambda $-martingale is a function $\mathcal{M}:\cs \rightarrow \mathbb{R}$ satisfying
			\[
				\lambda(\sigma) \mathcal{M}(\sigma) = \lambda (\sigma 1) \mathcal{M}(\sigma 1) + \lambda (\sigma 0) \mathcal{M}(\sigma 0).
			\]
			We now define the notion of a KL martingale, which will be able to select
			which bit it will bet on next, in a generalized Bernoulli measure.
			We use the notation from Definition \ref{def_KL_Stoc}.

			\begin{definition}[$\lambda$-KL randomness]
				Let $\lambda$ be a generalized Bernoulli measure.
				A $\lambda$-KL martingale is a pair $(f, \M)$
				where $f$ is a selection function $(\delta, n)$ and $\mathcal{M}$ is a function $\fs \rightarrow \R$
				such that, for every sequence $X \in \cs$ for which $f$ select infinitely many bits of $X$,
				\[
					\M \left(T_f^X (k) \right) =
					b_{i}^{\lambda} \M \left(T_f^X (k) .1 \right) + (1- b_{i}^{\lambda}) \M \left(T_f^X (k) .0 \right)
				\]
				for all $k \in \N$, where $i=n ( V_f^X (k))$.

				We say that $X$ is $\lambda$-KL random if, for every $\lambda$-KL martingale computable with oracle $(b_i^{\lambda})$,
				the sequence $(\M ( T_f^X (k)))$ is bounded.
				For a sequence $Y$, we say that $X$ is $\lambda$-KL$^Y$ random if this is true
				even when the $\lambda$-KL martingale is also given oracle access to $Y$.
			\end{definition}

		\subsection{Hippocratic stochasticity is not stochasticity}

			We will show that, despite the fact that we now allow non-monotone
			strategies, once again there exist sequences computable from $\alpha$
			which are $\alpha$-Hippocratic KL stochastic, for $\alpha \in \MLR \cap \Delta_2^0$
			(recall that Chaitin's constant $\Omega$ is the prototypical example of
			such an $\alpha$).

			We remark that
			our proof shows also that for $\alpha \in \MLR \cap \Delta_2^0$
			the Hippocratic and classical versions of von Mises--Wald--Church stochasticity are
			different (see \cite[Definition 7.4.1]{Downey:2010:ARC:1205766} for a formal definition).

			We first need a lemma:

			\begin{lemma}[\cite{MSU}, \cite{Downey:2010:ARC:1205766} p.311]\label{lemma_from_MLR_to_KLR}
				If $X$ is Martin-L\"of random for a computable generalized Bernoulli measure $\lambda$,
				then $X$ is $\lambda$-KL random.
			\end{lemma}
			\begin{proof}[see \cite{Downey:2010:ARC:1205766}]
				Consider the set of sequences in which the player achieves
				capital greater than $j$ when he started with capital $1$.
				For obvious reasons, this is an effective open set of measure less than $1/j$.
				\QED
			\end{proof}

			\begin{theorem}\label{Thm_hypo_KL_stoc}
				Let $\alpha \in \MLR \cap \Delta_2^0$.
				There exists a sequence $Q \in \cs$, computable from $\alpha$, such that
				$Q$ is $\alpha$-Hippocratic KL stochastic.
			\end{theorem}
			\begin{proof}
				We will first define the sequence $Q$, and then show that $Q$ is
				$\lambda$-KL random for a certain generalized Bernoulli measure $\lambda$
				for which the parameters $ ( b_i^\lambda)$ converge to $\alpha$.
				Finally we will show that it follows that $Q$ is actually $\alpha$-Hippocratic
				KL stochastic.

				Since $\alpha \in \Delta^2_0$,
				by Shoenfield's Limit Lemma $\alpha$ is the limit of
				a computable sequence of real numbers
				(although the convergence must be extremely slow, since $\alpha$ is not computable).
				In particular there exists a computable sequence of finite strings $(\beta^k)$
				such that $ \beta^k \in {\{ 0,1 \}}^k$ and
				\[
					\lim_{k \rightarrow\infty} \beta^k = \alpha.
				\]
				We define $Q_k$ by
				\[
					Q_k = \left\{
					\begin{array}{ll}
					1 & \mbox{ ~ if } 0.\beta^k \geq 0.\alpha_{k'+1} \dots \alpha_{k'+k} \\
					0 & \mbox{ ~ otherwise. }
					\end{array}
					\right.
				\]
				We set $Q = Q_1 Q_2 \dots$.
				Intuitively, as in the proof of Theorem \ref{computrandomness_thm},
				we are using $\alpha$ as a random bit generator to simulate a
				sequence of Bernoulli trials with parameter~$\beta^k $.

				Notice that the transformation mapping $\alpha$ to $Q$ is a total computable function.
				We know that in general if $g$ is total computable, and $X$ is (Martin-L\"of)
				random for the
				uniform measure $\mu$, then $g(X)$ is
				random for the measure $\mu \circ g^{-1}$ (see \cite{S86} for a proof of this fact).
				Since $\alpha \in \MLR$, in our case
				this tell us that $Q$ is
				random for exactly the generalized Bernoulli measure $\lambda$ given by
				$b_i^{\lambda} = \beta^i$.

				It follows from Lemma \ref{lemma_from_MLR_to_KLR}
				that $Q$ is $\lambda$-KL random. Finally by Lemma \ref{lemmaklrandtoklstoc}
				below we can conclude that $Q$ is $\alpha$-Hippocratic
				KL stochastic, completing the argument.
				\QED
			\end{proof}

			\begin{lemma}\label{lemmaklrandtoklstoc}
				Let $\lambda$ be a computable generalized Bernoulli measure and suppose
				\[
				\lim_{i\rightarrow\infty} b_i^{\lambda} =p.
				\]
				Then every $\lambda$-KL random sequence
				is $\mu_p$-Hippocratic KL stochastic.
			\end{lemma}
			\begin{proof}
				We prove the contrapositive. Without loss of generality we assume that $0 < p < 1/2$.
				Suppose that the sequence $X$ is not $\mu_p$-Hippocratic KL stochastic.
				Then there is a selection function $f$, computed without an oracle for $p$,
				for which the selected sequence $U^X_f$ is infinite and
				$\Phi(T^X_f(k))$ does not tend to the limit $p$.
				We will define a $\lambda$-KL martingale which wins on $X$.

				Without loss of generality,
				there is a rational number $\tau > 0$ such that \mbox{$p + 2\tau<1$} and
				\[
					\limsup_{k \rightarrow\infty} \Phi ( T_f^A (k)) \geq p + 2 \tau.
				\]
				Since $(b_k^{\lambda})$ converges to $p$, by changing the selection function
				if necessary, we may assume without loss of generality that
				$b_k^{\lambda} < p + \tau$ for all locations $k$ read by the selection function.
				We let
				\[
					\gamma = \frac{p+2 \tau}{p+\tau}-1 >0.
				\]
				We let $\delta$ be a rational in $(0,1)$ satisfying both
				\[
					\delta ~ \frac{{(1-p-\tau)}^2}{ {(p + \tau)}^2} \leq \tau
					\mathrm{ \qquad and \qquad}
					\log (1-\delta)> -\frac{\delta}{\ln 2} (1+\gamma /2)
				\]
				(where $\log$ is to base $2$).
				Such a $\delta $ exists because $\log (1-\delta)/ \delta $
				converges to $ -1/\ln 2$ as $\delta > 0 $ tends to 0.
				Note that $\delta (1-p-\tau) / (p+\tau) <1$.

				Let $\M$ be the $\lambda$-KL martingale which begins with capital $1$
				and then, using selection function $f$,
				bets every turn a fraction $ \delta$ of its current capital on the next bit being $1$.
				Formally, writing $T_k$ for $T^X_f(k)$ and $i$ for $n(V^X_f(k))$,
				we put $\M(\varepsilon) = 1$ and for each $k$
				\begin{align*}
					\M(T_k . 0) &= \M (T_k) \Big( 1-\delta \Big),\\
					\M(T_k . 1) &= \M(T_k) \left(1 -\delta + \frac{\delta}{ {b_i^\lambda}} \right)
					\ge \M(T_k) \left(1+ \delta \left(\frac{1}{p + \tau} -1\right) \right).
				\end{align*}
				We do not care how $\M$ is defined elsewhere.

				By induction,
				\[
					\M (T_k) \ge {\left(1+\delta ~ \frac{ 1-p-\tau}{ p+\tau} \right)}^{\# 1 ( T_k)}
					{\left(1- \delta \right)}^{\# 0 ( T_k)}
				\]
				and thus
				\[
					\frac{\log (\M (T_k))}{k}
					\geq \frac{\# 1 (T_k)}{k} \log \left(1+ \delta \, \frac{ 1-p-\tau}{ p+\tau} \right) +
					\frac{\# 0 ( T_k)}{k} \log \left(1- \delta \right).
				\]

				In the following we use standard properties of the logarithm together with our
				definitions of $\tau$, $\delta$ and $\gamma$. In particular, note that
				if we let $x= \delta (1-p-\tau) / (p+\tau)$ then $0<x<1$ so we have
				$\ln (1+x) > x-x^2/2 >x-x^2$. We have
				\begin{align*}
					\limsup_{k \rightarrow \infty} & \, \frac{\log \M(T_k)}{k}
					\ge
					(p+2 \tau) \log \left(1+ \delta \, \frac{ 1-p-\tau}{ p+\tau} \right) +
					(1-p - 2 \tau) \log \left(1- \delta \right) \\
					&\geq \frac{p+2 \tau}{\ln 2} \left(\delta \, \frac{ 1-p-\tau}{ p+\tau}
					- \delta^2 \frac{ (1-p-\tau)^2}{ (p+\tau)^2} \right) -
					\frac{\delta (1+\gamma/2)}{\ln 2} (1-p -2\tau) \\
					&\geq \frac{\delta}{\ln 2} \left((1-p-\tau)(1+\gamma)
					- \delta \, \frac{ (1-p-\tau)^2}{ (p+\tau)^2} - (1+\gamma/2)(1 -p - 2 \tau) \right) \\
					&\geq \frac{\delta}{\ln 2}
					\left(\tau +\frac{\gamma}{2} (1-p-\tau) -\delta \, \frac{ (1-p-\tau)^2}{ (p+\tau)^2} \right) \\
					&\geq \frac{\delta \gamma \, (1-p-\tau)}{2\ln 2} \\
					&> 0.
				\end{align*}
				Hence $\log(\M(T_k)) \ge ck$ infinitely often, for some strictly positive constant $c$. Therefore
				our martingale is unbounded on $U^X_f$.
				\QED
			\end{proof}

		\subsection{Kolmogorov--Loveland randomness}\label{KLR_section}
			We have shown that, for computable randomness and non-monotone stochasticity,
			whether a string is random can depend on whether or not we have access
			to the actual bias of the coin.
			It is natural to ask if this remains true for Kolmogorov--Loveland randomness.

			\begin{lemma}[\cite{Merkle05kolmogorov-lovelandrandomness}]\label{VLforKL}
				For sequences $X, Y \in \cs$,
				$X\oplus Y$ is $\mu_p$-KL random if and only if both
				$X$ is $\mu_p$-KL$^Y$-random and $Y$ is $\mu_p$-KL$^X$-random,
				and this remains true in the Hippocratic setting (that is, where the
				KL martingales do not have oracle access to $p$).
				\QED
			\end{lemma}

			The proof is a straightforward adaptation of the proof given in \cite[Proposition 11]{Merkle05kolmogorov-lovelandrandomness}.
			Using Lemma \ref{VLforKL} we can show an equivalence between our question and a statistical question.

			\begin{theorem}\label{two}
				The following two sentences are equivalent.
				\begin{enumerate}
					\item
					The $\mu_p$-Hippocratic KL random and $\mu_p$-KL random sequences are the same.
					\item
					From every $\mu_p$-Hippocratic KL random sequence $X$, we can compute $p$.
				\end{enumerate}
			\end{theorem}
			\begin{proof}
				For (1) $\Rightarrow$ (2), we know that if $X$ is $\mu_p$-KL random
				then it must satisfy the law of the iterated logarithm (see \cite{W00} for this result).
				Hence we know how quickly $\Phi (X \restrict k)$ converges to $p$ and
				using this we can (non-uniformly) compute $p$ from~$X$.

				For (2) $\Rightarrow$ (1), suppose that $X$ is $\mu_p$-Hippocratic KL random
				but not $\mu_p$-KL random. Let $X = Y \oplus Z$. Then by Lemma \ref{VLforKL},
				$Y$ (say) is not $\mu_p$-KL$^Z$ random, meaning that there is a KL martingale $(\M, f)$
				which, given
				oracle access to $p$ and $Z$, wins on $Y$. On the other hand, both $Y$ and $Z$
				remain $\mu_p$-Hippocratic KL random, so in particular by (2)
				if we have oracle access to $Z$ then
				we can compute $p$. But this means that we can easily convert $(\M, f)$
				into a $\mu_p$-Hippocratic KL martingale which wins on $X$, since to answer oracle
				queries to either $Z$ or $p$ it is enough to scan $Z$ and do some computation.
				\QED
			\end{proof}
			\begin{remark}
				We do not know what the shared truth value of the two sentences in Theorem \ref{two} is.
				We also do not know whether there is a $\mu_p$-Hippocratic computably random sequence to which $p$ is not Turing reducible.
				As an anonymous referee pointed out, this would give a stronger answer to our question ``How much randomness is needed for statistics?''
			\end{remark}

	\section*{Acknowledgements}
		The authors would like to thank Samuel Buss for his invitation to University of California, San Diego.
		Without his help and the university's support this paper would never exist.
		Taveneaux's research has been helped by a travel grant of the \emph{Fondation Sciences Math\'{e}matiques de Paris}.
		Kjos-Hanssen's research was partially supported by NSF (USA) grant no. DMS-0901020.
		Thapen's research was partially supported by grant IAA100190902 of GA AV \v{C}R,
		by Center of Excellence CE-ITI under grant P202/12/G061 of GA CR and RVO: 67985840
		and by a visiting fellowship at the Isaac Newton Institute for the Mathematical
		Sciences in the programme \emph{Semantics and Syntax}.

		An earlier version of this paper appeared in the proceedings of the
		Computability in Europe 2012 conference \cite{KTT_2012}.

	\bibliographystyle{plain}
	\bibliography{randomness_apal}
\end{document}